\documentclass{birkjour}

\usepackage{amssymb}
\usepackage{amsmath}
\usepackage{mathrsfs}
\usepackage{amsfonts}
\usepackage{latexsym}
\usepackage{color}
\usepackage{soul,cancel}

\newtheorem{theorem}{Theorem}
\newtheorem{cor}{Corollary}

\usepackage{rawfonts}

\newfont{\OOO}{cmr10 scaled 2986}
\newfont{\OO}{cmr10 scaled 1440}

    \usepackage[OT2,OT1]{fontenc}
    \newcommand\cyr{%
    \renewcommand\rmdefault{wncyr}%
    \renewcommand\sfdefault{wncyss}%
    \renewcommand\encodingdefault{OT2}%
    \normalfont
    \selectfont}
    \DeclareTextFontCommand{\textcyr}{\cyr}

\newcommand{\Z}{\mathbb{Z}}
\newcommand{\R}{\mathbb{R}}
\newcommand{\C}{\mathbb{C}}

\newcommand{\T}{\mathbb{T}}

\renewcommand{\L}{\mathbb{L}}

\newcommand{\eq} [1] {\begin{equation}\label{#1}\quad}
\newcommand{\en} {\end{equation}}

\newcommand{\norm}[1]{\left\Vert#1\right\Vert}
\newcommand{\abs}[1]{\left\vert#1\right\vert}

\newcommand{\ind}{\operatorname{Ind}}

\newcommand{\diag}{\operatorname{diag}}

\renewcommand{\Re}{\operatorname{Re}}

\newcommand{\re}{\operatorname{Re}}
\newcommand{\dist}{\operatorname{dist}}

\begin{document}
\title[]{A note on the factorization of \\ some structured matrix functions}

\author[I. Spitkovsky]{Ilya M. Spitkovsky}
\address{Division of Science\\
New York  University Abu Dhabi (NYUAD)\\
Saadiyat Island\\
P.O. Box 129188\\
Abu Dhabi\\
United Arab Emirates}
\email{ims2@nyu.edu, imspitkovsky@gmail.com}
\author[A. Voronin]{Anatoly F. Voronin}
\address{Sobolev Institute of Mathematics \\
Siberian Branch, Russian Academy of Sciences \\
Academician Koptyug Str. 4 \\
Novosibirsk 630090, Russia}
\email{voronin@math.nsc.ru}

\begin{abstract}
{Let $G$ be a block matrix function with one diagonal block {$A$ being positive definite} and the off diagonal blocks complex conjugates of each other.
Conditions are obtained for $G$ to be factorable (in particular, with zero partial indices) in terms of the Schur complement of $A$.}
\end{abstract}

\thanks{The first author was supported in part by Faculty Research funding from the Division of Science and Mathematics, New York University Abu Dhabi.}



\maketitle

\section{Preliminary results}

Let $\L$ be a simple closed curve in the complex plane $\C$.
Denote its interior and exterior domains by $D^+$ and $D^- (\ni \infty)$ respectively. The {\em Riemann-Hilbert boundary value
problem} consists in finding functions $\phi^\pm$ analytic
in $D^\pm$ by the condition
\eq{RH} \phi^+(t)=G(t)\phi^-(t)+g(t),\quad t\in \L, \en
imposed on its boundary values. Here $G$
and $g$ are known functions defined on $\L$.

In the vector version of \eqref{RH}, $\phi^\pm$ and $g$ are vector
functions with say $n$ entries while $G$ is an $n$-by-$n$ matrix
function.

We are interested in the $L_p$ setting of \eqref{RH}. This means
that $g\in L_p(\L)$, $\phi^\pm\in E_p^\pm:=E_p(D^\pm)$ and $G\in L_\infty(\L)$ (all inclusions for vector and
matrix functions here and below are understood entrywise). {We are using the notation $E_p(A)$ for the Smirnov classes
in the domain $A$, $0<p\leq\infty$. See e.g.\cite{Dur} for the definition and properties of these classes.
Note in particular that, in the case of $\L$ being the unit circle $\T$, $E_p^\pm$ become the classical Hardy
spaces $H_p^\pm$. }

It is known (see e.g. \cite{GKS03,LS}) that for a given $p\in (1,\infty)$ problem \eqref{RH} is Fredholm if and only if $G$ admits a representation
\eq{fac} G=G_+\Lambda G_-, \en
where
\[ G_+\in E_p^+,\ G_-\in E_q^-,\ G_+^{-1}\in E_q^+,
\ G_-^{-1}\in E_p^-,\quad q=p/(p-1),\]
$\Lambda(t)=\diag[(t-z_0)^{\kappa_1},\ldots,(t-z_0)^{\kappa_n}]$,
$\kappa_1,\ldots,\kappa_n\in\Z$, $z_0$ is an arbitrarily fixed
point of $D^+$, and $G_-^{-1}\Lambda S G_+^{-1}$,
with $S$ denoting the singular integral operator with the Cauchy
kernel, is bounded as an operator on $L_p(\L)$. Representation
\eqref{fac} satisfying all these conditions is sometimes called an
$L_p$-{\em factorization} of $G$.

Note the role of the {\em partial indices} $\kappa_1,\ldots,\kappa_n$: the number $\lambda$ of linearly independent solutions of the homogenous ($g=0$) problem \eqref{RH}
is the sum of the positive $\kappa_j$, while the number $\eta$
of linear constraints on $g$ under which the non-homogenous problem admits
a solution is opposite to the sum of the negative $\kappa_j$.
In particular, the {\em index} of problem \eqref{RH}, i.e.
the difference $\lambda-\eta$, equals the {\em total index}
$\kappa=\sum_{j=1}^n\kappa_j$ of factorization \eqref{fac}.

This justifies the continuing interest in finding explicit factorization criteria, as well as formulas for the partial indices,
for various classes of matrix functions. To describe one result
in this direction, pertinent to the content of this short note, we need to recall a few more notions.

The {\em numerical range} $W(A)$ (a.k.a. the {\em field of values}, or the {\em Hausdorff set}) of a square matrix $A$ is defined and denoted as
\eq{NR} W(A)=\{x^*Ax\colon x^*x=1\}, \en
{see e.g. \cite{GlaLyu06} or \cite{GusRa}.}

A matrix function $G$ is {\em $\alpha$-sectorial} on some subset $X$ of its domain
if for some sector $\mathcal S$ with the vertex at the origin and
the angle $\alpha$ we have $W(G(t))\subset\mathcal S$ a.e. on $X$. In its turn, $G$ is {\em locally}
$\alpha$-sectorial on $X$ if for every  $t\in X$ there is a neighborhood of $t$ on which $G$ is $\alpha$-sectorial.
The respective sector ${\mathcal S}_t$ may a priori depend on $t$.

Clearly, a matrix function $G$ defined on $\L$ is locally
$\alpha$-sectorial if and only if it has the form
\eq{gchi} G=\chi G_0,\en
where $\chi$ is a continuous on $\L$ and invertible function
while $G_0$ is $\alpha$-sectorial on $\L$. Although representation
\eqref{gchi} is not unique, the winding number $\ind\chi$ of the function  $\chi$ is defined uniquely. We will call it the
winding number of W(G), and denote $\ind G$.

{The following factorability condition in terms of the numerical range behavior
was obtained in \cite{Spit80}, see also \cite[Section 3]{GKS03} for other relevant results and the history of the
subject.}
\begin{theorem}\label{th:spit} Let $G$ be an invertible in $L_\infty$ $n$-by-$n$ locally $\alpha$-sectorial matrix function defined on a smooth simple closed curve $\L$. Then $G$ admits a representation \eqref{fac} delivering an $L_p$-factorization of $G$ for all
\eq{p} p\in\left(\frac{2\pi}{2\pi-\alpha}, \frac{2\pi}{\alpha}\right).\en
Moreover, the total index $\kappa$ of
\eqref{fac} equals $n\ind G$, while for $\L$ being a circle
we further have \eq{kj} \kappa_1=\cdots = \kappa_n=\ind G. \en  \end{theorem}
\begin{cor}\label{co:can} If in the case of a circle we in addition have $\ind G=0$, then the $L_p$-factorization
\eqref{fac} of $G$ {for $p$ satisfying \eqref{p}} is {\sl canonical}, i.e. $\kappa_j=0, \ j=1,\ldots,n$.\end{cor}

\section{Main statement}\label{s:main}

Let $G$ have the special structure
\eq{G} G =\begin{bmatrix} A & B^* \\ B & D\end{bmatrix}, \en
where {$A$ is an $m$-by-$m$ positive definite $L_\infty$-invertible matrix function while} $B, D$ are $k$-by-$m$ and $k$-by-$k$, respectively.

An important role in what follows is played by the {\em Schur complement} $\Gamma$ of the upper left block $A$ in $G$:
\eq{Ga} \Gamma=D-B{A^{-1}}B^*. \en
\begin{theorem}\label{th:can}Let $G$ be an invertible $L_\infty$ matrix function given by \eqref{G} and defined on a smooth simple closed curve $\L$. If the respective matrix function $\Gamma$ is locally $\alpha$-sectorial, with all the involved sectors ${\mathcal S}_t$ containing the positive ray, then $G$ admits an $L_p$-factorization \eqref{fac} with the zero total index, for all $p$ as in \eqref{p}. If in addition $\L$ is a circle, then factorization \eqref{fac} is canonical.
\end{theorem}
\begin{proof} The identity
\[ G= \begin{bmatrix} {A^{\frac{1}{2}}} & 0 \\ B {A^{-\frac{1}{2}}} & I\end{bmatrix}
\begin{bmatrix} I & 0 \\ 0 & \Gamma\end{bmatrix}
\begin{bmatrix}  {A^{\frac{1}{2}}} &  {A^{-\frac{1}{2}}}B^* \\ 0 & I\end{bmatrix} \]
implies that the matrices $G(t)$ and $\diag[I, \Gamma(t)]$
are congruent for all $t\in\L$.  The minimal sector with the vertex at the origin and
containing $W(G(t))$ and $W(\diag[I,\Gamma(t)])$ is therefore the
same. Since the latter numerical range is simply the convex hull of $W(\Gamma(t))$ and the point one,  we have
$W(G(\tau))\subset{\mathcal S}_t$ for $\tau$ from some neighborhood of $t$. In other words, the matrix function $G$
satisfies conditions of Theorem~\ref{th:spit}. This
guarantess the $L_p$-factorability of $G$ for all $p$ satisfying
\eqref{p}. Moreover, since the sectors ${\mathcal S}_t$ all contain the positive ray, they skip the negative one, implying $\ind G=0$. \end{proof}
This result is non-trivial even when the lower right block $D$ of \eqref{G} is one-dimensional, i.e., in the case of a scalar valued function $\Gamma$ when the (local) $\alpha$-sectoriality condition is imposed simply on its values.

Another simplification occurs when $\Gamma$ happens to be continuous.

\begin{cor}\label{co:cont}Let $G$ be an invertible $L_\infty$ matrix function given by \eqref{G}, defined on a smooth simple closed curve $\L$, and such that the respective matrix function $\Gamma$ is continuous on $\L$. Suppose that for some $\alpha<\pi$ and all $t\in\L$ the numerical range of
$\Gamma(t)$ and the positive ray both lie in the same sector ${\mathcal S}_t$ (depending on $t$) with the vertex at the origin and the angle $\alpha$. Then $G$ admits an $L_p$ factorization \eqref{fac} with the zero total index, for all $p$ as in \eqref{p}. If in addition $\L$ is a circle, then the factorization \eqref{fac} is canonical.
\end{cor}

Of course, $\Gamma$ is continuos if $G$ itself is continuous.
In this case, however, the $L_p$-factorability of $G$ is guranteed
for all $p\in (1,\infty)$ just by the invertibility of $G$, and the
total index $\kappa$ is simply the winding number of $\det G$. So,
the only interesting aspect of Theorem~\ref{th:can} in this setting concerns
the values of the partial indices for circular $\L$.
\begin{theorem}\label{th:cir} Let $G$ be a continuos matrix function of the form \eqref{G} defined on a circle.
Suppose that
\eq{strict} \{x\in\R\colon x\leq 0\}\cap W(\Gamma(t))=\emptyset \text{ for all } t\in\T.\en Then the
partial indices of $G$ are all equal to zero.
 \end{theorem}
\begin{proof}Since the numerical range is convex and compact,
condition \eqref{strict} implies that $W(\Gamma(t))$ lies in
some sector ${\mathcal S}_t$ with the angle less than $\pi$
and disjoint with the negative ray. Expanding the angle if necessary, we can still keep it under $\pi$ but have the positive ray covered. {In addition, $\Gamma(t)$ is invertible for all $t$ since $0\notin W(\Gamma(t))$, thus
implying the invertibility of $G$.} By Corollary~\ref{co:cont}, $G$ admits a canonical
$L_2$-factorization. The latter then serves as a canonical $L_p$-factorization of $G$ for all $p\in (1,\infty)$. 	
 \end{proof}

\section{Additional comments}

Sufficient conditions provided by Theorem~\ref{th:can} are far from being necessary{, even in the
special case $A=I$:
\eq{G1} G=\begin{bmatrix} I & B^* \\ B & D \end{bmatrix}. \en

} For $p=2$,
in particular, the following ``disjoint'' sufficient condition also
holds. Recall that $E^\pm_\infty+C$ (the algebraic sum of $E^\pm_\infty$ with the class $C$ of all functions continuous on $\L$) is closed in $L_\infty(\L)$ and is thus a subalgebra of $L_\infty(\L)$.

\begin{theorem}\label{th:neg} Let $G$ be an $L_\infty$ matrix function given by \eqref{G1}{, so that $\Gamma=D-BB^*$,} and defined on a simple closed smooth curve $\L$.

{\em (i)} If for some $k$-by-$m$ matrix function $B_1$ such that
$B-B_1\in E_\infty^++C$ the matrix function $\re\Gamma+B_1B_1^*$ is uniformly negative
on $\L$, then $G$ is $L_2$-factorable.

{\em (ii)} If in addtion $\L$ is a circle and $B-B_1\in E^+_\infty$, then all partial indices
of $G$ are equal to zero.  \end{theorem}
\begin{proof} Observe that $L_p$-factorability property
is preserved under multiplication on the left/right	by
a matrix function invertible in $E^\pm_\infty+C$. Moreover,
if these factors and their inverses are in fact in $E^\pm_\infty$ then the values of parital indices are also preserved.

Denoting $B_1-B=X$, we conclude therefore that $G$ and
\eq{-G} \begin{bmatrix} I & 0 \\ X & -I\end{bmatrix}G
 \begin{bmatrix} I & -X^* \\ 0 & I\end{bmatrix}=
\begin{bmatrix} I & B_1^* \\ -B_1 & -(\Gamma+B_1B_1^*)\end{bmatrix}	\en
are $L_p$-factorable only simultaneously in setting (i), and in addition also have the same sets of partial indices in setting (ii). But the real part of \eqref{-G} is
$\diag[I, -(\re\Gamma+B_1B_1^*)]$, and so uniformly positive under the condition imposed. A particular case of Theorem~\ref{th:spit} (corresponding to the fixed sector lying in the right half plane, {
and going back to the classical paper \cite{GoKr58})} is applicable. \end{proof}
Of course, when applying Theorem~\ref{th:neg} it makes sense to
choose $X$ as the best approximation to $B$ in {$E^+_\infty+C$ in setting (i), and } $E^+_\infty$
{in setting (ii)}; see \cite[Chapter 13]{Peller} for the pertinent discussion.

{\begin{cor}\label{co:k=1}Let in \eqref{G1} $k=1$, i.e. $B=[b_1,\ldots,b_m]$, $\L=\T$, and the (scalar valued) function {$D$} is such that \[ \sup_{t\in\T}{\left(\Re D(t)-\sum_{j=1}^m\abs{b_j(t)}^2\right)}<-\sum_{j=1}^m\dist^2(b_j,H^+_\infty). \]
Then $G$ admits a canonical $L_2$-factorization. \end{cor}}

To put things in perspective, consider the case when $\L$ is a circle and $D=BB^*+\gamma I$, $\gamma\in\C$. Then $\Gamma{=\gamma I}$
does not depend on $t$ and is a scalar multiple of the identity.
Corollary~\ref{co:cont} then implies that $G$ admits a canonical
$L_p$-factorization if
$\abs{\arg\gamma}<\frac{2\pi}{\max\{p,q\}}$. In particular, its
$L_2$-factorization exists and is canonical if $\gamma$ is not a
real non-positive number. In its turn, Theorem~\ref{th:neg} implies that a canonical $L_2$-factorization of $G$ is guaranteed
if $\gamma$ is negative and smaller than $-\norm{H_B}$, where
$H_B$ is the Hankel operator with the matrix symbol $B$. {(Recall that $H_B=P_-BP_+$,  where $P_\pm=\frac{1}{2}(I\pm S)$ are the
complimentary orthogonal projections associated with the self-adjoint involution $S$ and acting entry-wise, and that $\dist(B,H^+_\infty)=\norm{H_B}$.)}
The truth of the matter is, however, that such $G$ {is $L_2$-factorable if and only if $-\gamma$ does not belong to the essential} spectrum of $H_BH_B^*$,
{and this factorization is canonical unless $-\gamma\in\sigma(H_BH_B^*)$. This result (for $k=m$) was established in \cite{CoKraTe2}, preceded by its scalar ($k=m=1$) version in \cite{CoKraTe1}.
Note that the latter case for $\gamma=-1$ (which by scaling covers all $\gamma<0$) goes back to \cite{LitSpit80,LitSpit82}.}

\providecommand{\bysame}{\leavevmode\hbox to3em{\hrulefill}\thinspace}
\providecommand{\MR}{\relax\ifhmode\unskip\space\fi MR }
\providecommand{\MRhref}[2]{%
  \href{http://www.ams.org/mathscinet-getitem?mr=#1}{#2}
}
\providecommand{\href}[2]{#2}

\end{document}